\documentclass[10pt]{article}
\usepackage{amsthm}
\usepackage{latexsym}
\usepackage{amsmath}
\usepackage{amssymb}
\usepackage{epsfig}

    \newtheorem{thm}{Theorem}[section]
    \newtheorem{prop}[thm]{Proposition}
    \newtheorem{lemma}[thm]{Lemma}
    
    \newtheorem{cor}[thm]{Corollary}

    \newtheorem{defn}[thm]{Definition}

    \newtheorem{rem}[thm]{Remark}
    \newtheorem{example}[thm]{Example}

\begin{document}

\title{Cellular resolutions of cointerval ideals}

\author{Anton Dochtermann\footnote{Supported by an Alexander von Humboldt postdoctoral fellowship.}\\ Department of Mathematics \\ Dartmouth College \\ Hanover, NH 03755 \\ {\tt anton.dochtermann@gmail.com }
    \and Alexander Engstr\"{o}m\footnote{Alexander Engstr\"{o}m is a Miller Research Fellow 2009-2012 at UC Berkeley, and gratefully acknowledges support from the Adolph C. and Mary Sprague Miller Institute for Basic Research in Science.} \\  Department of Mathematics \\ U.C. Berkeley \\
Berkeley, CA 94720 \\ {\tt alex@math.berkeley.edu}}

\date\today

\maketitle

\abstract{Minimal cellular resolutions of the edge ideals of \emph{cointerval} hypergraphs are constructed.  This class of $d$--uniform hypergraphs coincides with the complements of interval graphs (for the case $d=2$), and strictly contains the class of `strongly stable' hypergraphs corresponding to pure shifted simplicial complexes.  The polyhedral complexes supporting the resolutions are described as certain spaces of directed graph homomorphisms, and are realized as subcomplexes of mixed subdivisions of the Minkowski sums of simplices.  Resolutions of more general hypergraphs are obtained by considering decompositions into cointerval hypergraphs.}

\section{Introduction} \label{sec:intro}

An \emph{edge ideal} $I_H$ is an ideal in a polynomial ring generated by squarefree monomials of a fixed degree $d$ (the generators can be thought of as edges of a $d$-uniform hypergraph $H$, hence the name).  The study of edge ideals has recently enjoyed a surge of activity, and the most well-known results in this area relate algebraic properties of edge ideals to the combinatorial structure of the underlying (class of) graphs.

In this paper we study \emph{resolutions} of edge ideals, and in particular give explicit descriptions of minimal cellular resolutions for edge ideals of a large class of hypergraphs.  Given any ideal $I$ in a polynomial ring $S = k[x_1, \dots, x_n]$, a \emph{resolution} of $I$ is an exact chain complex of free $S$-modules describing the generators, the relations, the relations among the relations (and so on) of the ideal $I$.  A \emph{cellular} resolution encodes these modules and maps as the chain complex computing the homology of a labeled polyhedral complex.

Our main result is the construction and explicit embedding of a labeled polyhedral complex $X_H$ which supports a minimal cellular resolutions of the edge ideal of $H$, whenever $H$ is what we call a \emph{cointerval} hypergraph.  We refer to Theorem \ref{thm:resCoInt} and Proposition \ref{prop:embedding} for precise formulations.  The class of cointerval $d$--graphs can be seen to coincide with the complements of interval graphs (for the case $d=2$), and in general strictly contains the class of `strongly stable' hypergraphs corresponding to pure shifted simplicial complexes.  Hence our constructions can be seen as an extension of the results of Corso and Nagel \cite{CN,CN1} and Nagel and Reiner \cite{NR}, where cellular resolutions of strongly stable edge ideals (and related ideals) are considered.  Our constructions are also somewhat more explicit, in the sense that we obtain particular geometric embeddings of the complexes $X_H$.  In particular we realize each $X_H$ as a subcomplex of a certain \emph{mixed subdivision} of a dilated simplex.

The facial structure of $X_H$ is given by simple graph-theoretic data coming from the hypergraph $H$, and this allows us to provide transparent descriptions certain algebraic invariant including Betti numbers, etc., of $H$.  Furthermore, we can use the explicit description of the complexes to provide (not necessarily minimal) cellular resolutions of arbitrary hypergraphs by considering decompositions into cointerval graphs.  Our results in this area are all independent of the characteristic of the coefficient field.

The rest of the paper is organized as follows.  In Section \ref{sec:Def} we review the basic definitions relating to edge ideals and cellular resolutions.  In Section \ref{sec:complex} we describe the complexes $X_H$ that will support our resolutions, and establish some results regarding their topology.  We provide the definition of cointerval graphs in Section \ref{sec:resolution}.  Here we also state and prove our main result, namely that the complex $X_H$ supports a minimal cellular resolutions of the ideal $I_H$, whenever $H$ is a cointerval hypergraph.  In Section \ref{sec:mixed} we describe how the complexes $X_H$ can be realized as subcomplexes of certain well-studied mixed subdivisions of dilated simplices.  We consider resolutions of more general hypergraph edge ideals in Section \ref{sec:gluing}, and show how decompositions into cointerval subgraphs $H = H_1 \cup H_2$ leads to cellular resolutions obtained by gluing together the associated $X_{H_i}$.  Here we also provide a thorough analysis of all 3-graphs on at most 5 vertices to illustrate our methods.  We end in Section \ref{sec:further} with some comments regarding open questions and further study.

\section{Definitions} \label{sec:Def}

We briefly discuss the main objects involved in our study.  We begin with some graph-theoretic notions.  For a finite subset $V \subseteq \mathbb{Z}$, a \emph{(uniform) $d$--hypergraph} (or simply \emph{$d$--graph}) $H$ with \emph{vertex set $V(H) = V$} is a collection of subsets of $V$ (called edges), each of which has cardinality $d$.  We will often take $V = [n] := \{1, \dots, n\}$ and will suppress set notation in describing our edges, so that e.g. $245$ will denote the edge $\{2,4,5\}$.  The \emph{complete} $d$--hypergraph $K^d_n$ is the $d$--hypergraph on $[n]$ consisting of all possible $d$--subsets.  Note that by definition our graphs come with integer labels on the vertices.  If we want to consider the underlying (unlabeled) graph we will emphasize this distinction.  In dealing with $d$--graphs we will often be interested in considering induced $(d-1)$-subgraphs in the following sense.

\begin{defn}
Let $H$ be a $d$--graph and let $v \in V(H) \subseteq \mathbb{Z}$ be some vertex. Then the \emph{$v$--layer} of $H$ is a $(d-1)$--graph on $V\setminus v$ with edge set
\[ \{ v_1v_2\cdots v_{d-1} \mid vv_1v_2\cdots v_{d-1} \in E(H)\textrm{ and }v<v_1,v_2,\ldots ,v_{d-1} \}. \]
\end{defn}

Note that if $H$ is a 2-graph then one can think of the $v$--layer as simply the entries to the right of the $v$-column in the $v$-row of the adjacency matrix of $H$, the `edges' of the resulting 1-graph are simply the entries that have a nonzero entry.

If $W \subseteq V(H)$ is a subset of the vertices of a $d$-graph $H$, the \emph{induced} subgraph on $W$, denoted $H[W]$ (or sometimes simply $W$ if the context is clear) is the $d$-graph with vertex set $W$ and edges $\{E \subseteq W: E \in E(H)\}$.

We next turn to the algebraic notions, and refer to \cite{MS} for undefined terms and further discussion.  Throughout the paper we let $k$ denote a field, our results will be independent of the characteristic.  Given a $d$--graph $H$ on the vertices $V(H) = \{v_1, \dots, v_n\}$, the \emph{edge ideal} $I_H$ is by definition the monomial ideal in the polynomial ring $k[x_{v_1}, \dots, x_{v_n}]$ generated by the monomials corresponding to the edges of $H$,
\[ I_H = \langle \prod_{j=1}^d x_{{v_i}_j} \mid v_{i_1}v_{i_2}\cdots v_{i_d} \in E(H) \rangle. \]
\noindent
We will usually take $V(H) = [n]$, so that $S = k[x_1, \dots, x_n]$, but it will be convenient to have the more general setup as well.

We will sometimes employ the Stanley-Reisner theory of face rings of simplicial complexes, and in this context we let $\Delta$ denote a simplicial complex on the vertices $[n]$.  The \emph{Stanley-Reisner ideal} of $\Delta$, which we denote $I_{\Delta}$, is the ideal in $S$ generated by all monomials $x_{\sigma}$ corresponding to nonfaces $\sigma \notin \Delta$.   We let $R_{\Delta} = S/I_{\Delta}$, and recall that $\dim R_{\Delta}$, the (Krull) dimension of $R_{\Delta}$, is equal to $\dim(\Delta) + 1$.  We point out that the edge ideal of a $d$--hypergraph is the special cases of a Stanley-Reisner ideal generated in a fixed degree $d$.  We recover the simplicial complex $\Delta$ as $Ind(H)$, the \emph{independence complex} of the hypergraph $H$.

As monomial ideals, the edge ideals $I_H$ are endowed with a fine ${\mathbb Z}^n$--grading coming from the ${\mathbb Z}^n$--grading on $S$.  We will sometimes abuse notation and use $\alpha \in {\mathbb Z}^n$ to denote both a monomial degree (i.e. a vector in ${\mathbb Z}^n$), as well as a monomial with that degree.  For example, if $n = 6$ and if $235$ is an edge in $H$, the corresponding monomial $x_2x_3x_5$ will be regarded as the vector $(011010) \in {\mathbb Z}^6$.  In this paper we will be interested in finely graded resolutions of the $S$-module $I_H$.  If

\[{\mathcal F}: 0 \rightarrow \displaystyle{\bigoplus_{\alpha} S[-\alpha]^{\beta_{\ell, \alpha}}} \rightarrow \cdots
\rightarrow \displaystyle{\bigoplus_{\alpha} S[-\alpha]^{\beta_{0, \alpha}}} \rightarrow I_H \rightarrow 0 \]

\noindent
is a \emph{minimal} free resolution of $I_H$, then for $i \in {\mathbb N}$ and $\alpha \in {\mathbb Z}^n$ the numbers $\beta_{i,\alpha}$ are independent of the resolution and are called these \emph{finely graded Betti numbers} of $I_H$.  The \emph{coarsely graded} Betti numbers are of $I_H$ are given by $\beta_{i,j} = \sum_{|\alpha|=j} \beta_{i,\alpha}$.  The number $\ell$ (the length of a minimal resolution) is called the \emph{projective dimension} of $I_H$, which we will denote $\textrm{pdim}(I_H)$.  One can check that $\textrm{pdim}(S/I_H) = \textrm{pdim}(I_H) + 1$, and by the Auslander-Buchsbaum formula, we have $\dim S - \textrm{depth}(S/I_H) = \textrm{pdim}(S/I_H)$.  The ideal $I_H$ is said to have a \emph{$d$--linear} resolution if $\beta_{i,j} = 0$ whenever $j-i \neq d-1$.  A ring $R = S/I$ is \emph{Cohen-Macaulay} if $\dim R = \textrm{depth}\, R$.

\begin{rem}
As is typical in this area, when dealing with edge ideals of graphs we will often say that $H$ has a certain algebraic property (e.g., `$H$ has a linear resolution'), by which we mean that the edge ideal $I_H$ has this property.
\end{rem}

We will be interested in resolutions of the edge ideals $I_H$ which are supported on geometric complexes.  Given an oriented polyhedral complex $X$ with monomial labels on the faces, one constructs ${\mathcal F}_X$, a free graded chain complex of $S$-modules which computes the cellular homology of $X$.  Under certain circumstances (see Proposition \ref{prop:cellres}) this algebraic complex is a resolution of the ideal generated by the monomials corresponding to the labels of the vertices.  This notion of a \emph{cellular resolution} was introduced by Bayer and Sturmfels in \cite{BS} and generalizes several well-known resolutions of monomial ideals including the Taylor resolution and the Hull resolution.  We will often use the following criteria (taken from \cite{MS}) as a way to check whether a labeled complex supports a cellular resolution of the associated ideal.  Here for any $\alpha \in {\mathbb Z}^n$ we use the notation $X_{\leq \alpha}$ to denote the subcomplex of $X$ induced by those faces with monomial labels which divide $\alpha$.

\begin{prop} \label{prop:cellres}
Suppose $X$ is a complex with vertices labeled by monomials, and label the higher dimensional faces $F$ with $\textrm{lcm} \{\ell(v):v \in F\}$, the least common multiple of the labels $\ell(v)$ on the vertices $v$ of $F$.  Then the cellular free complex ${\mathcal F}_X$ is a cellular resolution if and only if $X_{\leq \alpha}$ is acyclic over $k$ for all $\alpha \in {\mathbb Z}^n$, in which case it is a free resolution of the ideal generated by all monomials corresponding to the vertex labels.  Furthermore, the resolution is minimal if whenever $F \subsetneq G$ is a strict inclusion of faces, the monomial labels on those faces differ.
\end{prop}

Also from \cite{MS} we have the following.  For any $\alpha \in {\mathbb Z}^n$ we here use $X_{< \alpha}$ to denote the subcomplex of $X$ given by all faces with labels strictly less than $\alpha$.

\begin{prop} \label{prop:cellbetti}
Let $X$ be a cellular resolution of an ideal $I$, For $i \geq 1$ and $\alpha \in {\mathbb Z}^n$ the finely graded Betti numbers of $I$ are given by
\[ \beta_{i, \alpha} = H_{i-1} (X_{< \alpha}; k).\]
\end{prop}

If a labeled complex $X$ supports a \emph{minimal} resolution of an ideal $I$, then for any $i \in {\mathbb N}$ and $\alpha \in {\mathbb N}^n$ the Betti numbers $\beta_{i,\alpha}$ can read off from the labeled complex directly.  This follows from the fact that each $i$-face of $X$ with label $\alpha$ contributes a term $S[-\alpha]$ in homological degree $i$ to the complex ${\mathcal F}_X$.

\section{The labeled complex and some properties} \label{sec:complex}

In this section we associate a polyhedral complex $X_H$ to any $d$-graph.  In what follows, a simplex with vertex set $V$ is denoted $\Delta_V$.  Also, for subsets $\sigma, \tau \subseteq V \subseteq {\mathbb Z}$, we use $\sigma < \tau$ to denote $s < t$ for all $s \in \sigma$ and $t \in \tau$.

\begin{defn}~\label{def:Xsupp}
Let $H$ be a $d$--graph on a finite vertex set $V \subseteq \mathbb{Z}$. The polyhedral complex $X_H$ is the
subcomplex of the product
\[ \prod_{i=1}^d \Delta_{V} \] satisfying
\begin{itemize}
\item[(1)] The vertices of $X_H$ are $v_1 \times v_2 \times \dots v_d$, where $v_1v_2 \cdots v_d$ is an edge of $H$;

\item[(2)] For $\sigma_i \subseteq V$, the cells $\sigma_1 \times \sigma_2 \times \cdots \times \sigma_d$ satisfy $\sigma_1 < \sigma_2 < \cdots < \sigma_d$.
\end{itemize}
\end{defn}

An example of the construction is given in Figure \ref{fig:complex}.  From Definition \ref{def:Xsupp} one can see that the dimension of a cell $F = \sigma_1 \times \sigma \times \dots \times \sigma_d$ in $X_H$ is given by $n-d$, where $n = |\sigma_1 \cup \dots \cup \sigma_d|$.  Here $X_H$ is defined to be a subcomplex of a rather large ambient space; in Section \ref{sec:mixed} we will see a more convenient embedding.

\begin{figure}[ht]
\begin{center}
  \includegraphics[scale=0.5]{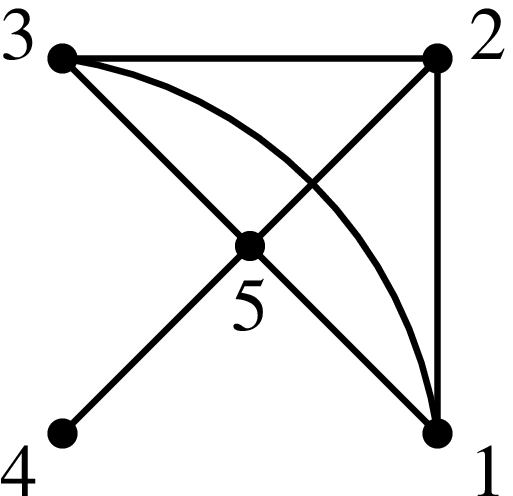} \quad \quad
   \includegraphics[scale=0.5]{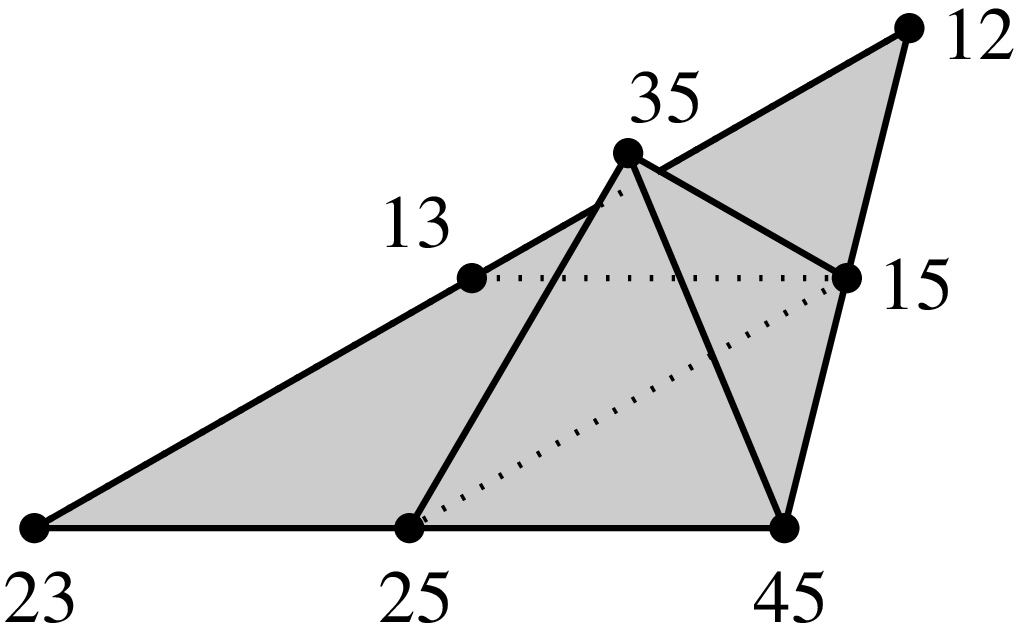}
  \caption{A graph $H$ along with the complex $X_H$.}\label{fig:complex}
\end{center}
\end{figure}

\begin{rem}\label{rem:labelcomplex}
For any $d$-graph $H$, the faces of the complex $X_H$ are naturally labeled by monomials.  In particular, the vertices are labeled by monomials corresponding to the edges of $H$ (i.e. the generators of $I_H$), and the higher dimensional faces $F = \sigma_1 \times \sigma_2 \times \cdots \times \sigma_d$ are labeled by
\[ \prod_{i=1}^d \prod_{v\in \sigma_i} x_{v}, \]
\noindent
which can be seen to equal the least common multiple of the monomial labels on the vertices of $F$.  As we remarked above, these monomials will sometimes be considered as vectors in ${\mathbb Z}^n$.
\end{rem}

\begin{rem}
Viewing $H$ as a directed $d$-graph (with orientation on the edges given by the integer labels on the vertices), one can regard $X_H$ as a `space of directed edges' of $H$.   If we let $E$ denote the $d$-graph with vertex set $[d]$ consisting of a single edge $\{1,2,\dots,d\}$, then $X_H = \textrm{Hom}(E,H)$, a space of directed graph homomorphisms from $E$ to $H$ analogous to the undirected $\textrm{Hom}$ complexes of \cite{BK}.  We will have more to say regarding this perspective in Section~\ref{sec:further}.
\end{rem}

In dealing with the topology of $X_H$ it will often be convenient to work with its face poset, where tools from poset topology can be applied to determine homotopy type, etc.  Since the order complex of the face poset of a polyhedral complex coincides with its barycentric subdivision, we are not losing any topological information.  We record this as a proposition.

\begin{prop}~\label{prop:Xsupp}
Let $H$ be a $d$-graph on a finite vertex set $V \subseteq \mathbb{Z}$. Define $P_H$ as the poset
of all maps
\[ \phi : \{1,2,\ldots, d\} \rightarrow 2^V \setminus \emptyset \]
such that
\begin{itemize}
\item[(1)] if $v_i\in \phi(i)$ for all $1\leq i \leq d$, then $v_1v_2\cdots v_d$ is an edge of $H$;
\item[(2)] if $1\leq i<j\leq d$, then $\phi(i) < \phi(j)$;
\end{itemize}
and $\phi \geq \nu$ if $\phi(i) \supseteq \nu(i)$ for all $i$. Then $P_H$ is the face poset of $X_H$ and the order complex
of $P_H$ is the barycentric subdivision of $X_H$.
\end{prop}

\begin{proof}
Clear.
\end{proof}

\subsection{Recursive topology and a folding lemma}
We next turn to establishing certain properties of the complexes $X_H$.  Our ultimate goal is to show that $X_H$ supports a cellular resolution whenever $H$ is cointerval, but we collect the necessary topological results in this section.

\begin{prop}
Let $v \in V(H) \subseteq \mathbb{Z}$ be the smallest vertex of $H$, and let $G$ be the $v$--layer of $H$.  If $v$ is in every edge of $H$, then $X_H$ and $X_G$ are isomorphic.
\end{prop}
\begin{proof}
Use the cellular isomorphism $v\times \sigma_2 \times \sigma_3 \times \cdots \times \sigma_d \rightarrow \sigma_2 \times \sigma_3 \times \cdots \times \sigma_d$.
\end{proof}

Our next result show that a certain deformation of a graph $H$ (related to a neighborhood containment of vertices) induces a homotopy equivalence of the associated complexes.  Readers familiar with $\textrm{Hom}$ complexes will see the similarity to the `folds' of graphs which are important in that context (see for instance \cite{BKproof}).  If $G$ is the $l$--layer of $H$, then we denote the edges corresponding to $H$ in $G$ with $l \ast H$.

\begin{thm}\label{thm:fold}
Let $H$ be a $d$--graph with vertex set $V\subseteq \mathbb{Z}$.  Suppose $i<j$ are vertices of $H$, and let $G$ be the $j$--layer, and $G'$ the $i$--layer, of $H$. If $G$ is a subgraph of $G'$ then the complexes $X_H$ and $X_{H\setminus j\ast G}$ are homotopy equivalent.
\end{thm}

\begin{proof}
We construct a poset $P'_H$ along with two monotone surjective poset maps
\[ P_H \rightarrow P'_H \rightarrow P_{H\setminus j\ast G}, \]
which will show that the order complexes of $P_H$, $P'_H$, and $P_{H\setminus j\ast G}$ are all homotopy equivalent.
Let
\[ P_H' = \{ \phi \in P_H \mid  \textrm{if $j\in \phi(1)$ then $i\in \phi(1)$} \} \]
be a subposet of $P_H$. Define the map
\[ \xi : P_H \rightarrow P_H' \]
by $\xi(\phi)=\phi$ if $j \not\in \phi(1)$, and by
\[ \xi(\phi)(l) = \left\{
\begin{array}{cl}
\phi(1)\cup\{i\} & \textrm{if $l=1$} \\
\phi(l) & \textrm{if $l \neq 1$}
\end{array}
\right.
\]
if $j \in \phi(1)$.

We need to check that the map $\xi$ is well-defined. If $j \not\in \phi(1)$ then $\xi(\phi)=\phi\in P_H'$.  For $j \in \phi(1)$ the conditions in Proposition~\ref{prop:Xsupp} need to be checked. The addition of $i$ to $\phi(1)$ satisfies the second condition since $i<j$, and also satisfies the first condition since the $j$--layer is a subgraph of the $i$--layer. It's clear that $\xi$ is monotone and surjective.

Let
\[ P''_H =  \{ \phi \in P_H' \mid  j \not \in \phi(1) \} \]
be a subposet of $P_H'$ and note that $P''_H$ and $P_{H\setminus j\ast G}$ are isomorphic as posets. Define the map
\[ \xi' : P'_H \rightarrow P''_H \]
by $\xi'(\phi)=\phi$ if $j\not\in \phi(1)$, and by
\[
\xi'(\phi)(l) = \left\{
\begin{array}{cl}
\phi(1)\setminus \{j\} & \textrm{if $l=1$} \\
\phi(l) & \textrm{if $l \neq 1$}
\end{array}
\right.
 \]
if $j \in \phi(1)$. The only obstruction to the map $\xi'$ being well-defined is if $\xi'(\phi)(1)= \emptyset$ for some
$\phi$. But $\phi \in P_H'$, so every $\phi(1)$ that contains $j$ also contains $i$. The map $\xi'$ is clearly both surjective and
monotone.
\end{proof}

We point out that since the surjective map $P_H\rightarrow P_{H\setminus j\ast G}$ is a composition of monotone maps, the induced map on the order complex of the underlying posets (which is the barycentric subdivision of the complexes $X_H$ and $X_{H\setminus j\ast G}$) is a collapsing and in particular a simple homotopy equivalence.

\section{Cointerval graphs and their cellular resolutions} \label{sec:resolution}
In this section we establish our main result, namely that the complexes $X_H$ support minimal cellular resolutions whenever $H$ is a cointerval graph. We discuss some consequences regarding combinatorial interpretations of the Betti numbers of cointerval graphs.

\subsection{Cointerval graphs}
We begin with the definition of cointerval graphs.

\begin{defn}\label{def:intHyp}
The class of \emph{cointerval} $d$--graphs is defined recursively as follows.

Any $1$--graph is cointerval. For $d>1$, the finite
$d$--graph $H$ with vertex set $V(H)\subseteq \mathbb{Z}$ is \emph{cointerval}
if
\begin{itemize}
\item[(1)] for every $i \in V(H)$ the $i$--layer of $H$ is cointerval;
\item[(2)] for every pair $i<j$ of vertices, the $j$--layer of $H$ is a subgraph of the $i$--layer of $H$.
\end{itemize}
\end{defn}

When $d = 2$ the class of cointerval graphs defined here can be seen to coincide with the well-studied \emph{complements of interval graphs} of structural graph theory (hence the name).  By definition, an \emph{interval graph} is a 2-graph with vertices given by intervals $I$ in the real line, and with adjacency $I \sim I^\prime$ if and only if $I \cap I^\prime \neq \emptyset$.  The \emph{complement} of a 2-graph $H$ is a 2-graph $H^C$ with the same vertex set as $H$ with adjacency $v \sim v^\prime$ in $H^C$ if and only if $v$ and $v^\prime$ do \emph{not} form an edge in $H$ (note that $v \neq v^\prime$, all graphs considered here do not have loops).  Given a cointerval 2-graph as in Definition \ref{def:intHyp}, one can obtain an interval representation as follows.  Without loss of generality, assume that $V(H) = [n]$.  To each vertex $i \in [n]$ assign the interval $[\ell_i+1,i]$, where $\ell_i$ is the largest neighbor of $i$ in $H$ such that $\ell_i < i$; assign $[1,i]$ to the vertex if there is no such $\ell_i$ (in particular assign $[1,1]$ to the vertex $1$).

Conversely, suppose $H$ is represented as the complement of an interval graph (so the vertices are given by intervals in the real line, with disjoint interval determining adjacency).  Order the intervals according to the rightmost endpoint, so that $[a,b] < [a^\prime, b^\prime]$ if $b < b^\prime$.  One can check that this determines a cointerval graph as in Definition \ref{def:intHyp}.

Cointerval $d$-graphs include the class of `strongly stable' hypergraphs, considered for instance in an algebraic context in \cite{NR}.  By definition a \emph{strongly stable} $d$--hypergraph $H$ on a vertex set $[n]$ has the property that whenever $E$ is an edge in $H$ with $i \in E$, then $E\backslash\{i\} \cup \{i-1\}$ is also an edge (whenever that set has the proper size).  These are also called `shifted' hypergraphs, or `square-free order ideals' in the Gale order on $d$-subsets.  When $d =2$, strongly stable 2-graphs correspond to the well-known class of `threshold' graphs.  Note that if in Definition \ref{def:intHyp} we required that 1-graphs had the property that whenever $i \in E(H)$ then $j \in E(H)$ for all $j < i$ we would recover the class of strongly stable hypergraphs.  An example of a threshold (and hence cointerval) 2-graph is depicted in Figure \ref{fig:ex1}.

\begin{figure}[ht]
\begin{center}
  \includegraphics[scale=0.5]{mp7.eps} \quad \quad
   \includegraphics[scale=0.5]{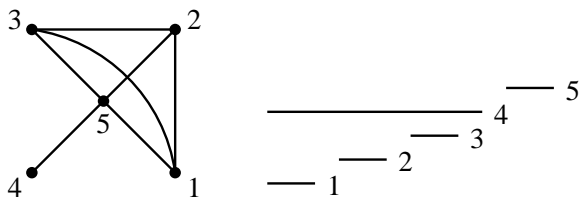}
  \caption{A threshold graph and its interval representation.}\label{fig:ex1}
\end{center}
\end{figure}

In light of Proposition \ref{prop:cellres}, the construction of our cellular resolutions will rely on the fact that our class of graphs is closed under taking induced subgraphs.  Our next results shows that this is indeed the case for cointerval graphs.

\begin{prop} \label{prop:indInt}
Any induced subgraph of a cointerval $d$--graph is a cointerval $d$--graph.
\end{prop}

\begin{proof}
We prove that if the $d$--graph $H$ is cointerval, then $H[V]$ is cointerval.

The proof is by induction on $d$. Any $1$--graph is cointerval, and hence assume $d>1$. We need to check the conditions in Definition~\ref{def:intHyp}.
\begin{itemize}
\item[(1)] Let $G$ be the $i$--layer of $H[V]$, and let $G'$ be the $i$--layer of $H$. Then by definition, $G=G'[V]$. The $(d-1)$--graph $G'$ is cointerval since it is a layer of $H$. By induction every induced subgraph
of $G'$ is cointerval.
\item[(2)]
 If $i<j$ are vertices of $H[V]$, then the $j$--layer of $H[V]$ is a subgraph of the $i$--layer of $H[V]$, since
the $j$--layer of $H$ is a subgraph of the $i$--layer of $H$.
\end{itemize}
\end{proof}

\subsection{Minimal cellular resolution of cointerval graphs}
In this section we establish our main result, Theorem \ref{thm:resCoInt}.  For the proof we will need the following observation.

\begin{lemma}\label{lem:coIntCont}
If $H$ is a non-empty cointerval $d$--graph with $d>0$, then $X_H$ is contractible.
\end{lemma}

\begin{proof}
We suppose $H$ is a $d$--graph and $t$ is the number of non-empty layers of $H$.  The proof is by induction over $d$ and $t$.

If $d=1$ then $X_H$ is a simplex and contractible. Now assume that $d>1$.

If $t=1$ and the $i$--layer $G$ is the only non-empty one, then $X_H$ and $X_G$ are isomorphic, and $X_H$ is contractible
by induction on $d$. Now assume that $t>1$.

Let $j$ be the maximal number such that the $j$--layer $G$ of $H$ is non-empty, and let $i<j$ be such that the $i$--layer of
$H$ is non-empty. The $d$--graph $H$ is cointerval, and hence by definition the $j$-layer is a subgraph of the $i$--layer. By Theorem~\ref{thm:fold}
the space $X_H$ is homotopy equivalent to $X_{H \setminus j \ast G}$. The $d$--graph $H \setminus j \ast G$ is also
cointerval, but with $t-1$ non-empty layers. By induction  $X_{H \setminus j \ast G}$ is contractible, and hence so is
$X_H$.
\end{proof}

With these tools in place we can state and prove our main result.  For this recall from Remark \ref{rem:labelcomplex} that $X_H$ is a labeled polyhedral complex with vertex labels corresponding to the monomial generators of the ideal $I_H$.

\begin{thm}\label{thm:resCoInt}
Let $H$ be a cointerval $d$--graph. Then the polyhedral complex $X_H$ supports a minimal cellular resolution of the edge ideal $I_H$.

In particular, for $i\geq 0$ and $\alpha \subseteq V(H)$ the graded Betti numbers are given by $\beta_{i,\alpha}(I_H)=0$ if $i \neq |\alpha|-d-1$, and
\[ \beta_{|\alpha|-d-1,\alpha}(I_H)=| \{ \sigma \in X_{H[\alpha]} \mid \dim \sigma = |\alpha|-d-1 \}|. \]
\noindent
In other words, the $(i,\alpha)$--Betti numbers are given by the number of faces of dimension $i$ in $X_H$ with monomial label $\alpha$.
\end{thm}

\begin{proof}
We will apply the conditions from Proposition \ref{prop:cellres}.  In particular let $n = |V(H)|$, and for any $\alpha \in {\mathbb Z}^n$ we consider the complex $(X_H)_{\leq \alpha}$.  All labels are square-free, so it is enough to restrict to $\alpha \in \{0,1\}^n$.  For any such $\alpha$, the complex $(X_H)_{\leq \alpha}$ is given by the
complex $X_{H[V]}$ where $V=\{ v \in V(H) \mid \alpha_v=1 \}$.  We are assuming that $H$ is cointerval, and hence by Proposition~\ref{prop:indInt} so is $H[V]$.  By Lemma~\ref{lem:coIntCont} the complex $X_{H[V]}$ is contractible, and hence by Proposition \ref{prop:cellres} the complex $X_H$ supports a cellular resolution of $I_H$.

We note that if $\sigma \subsetneq \tau$ is a strict containment of faces then the monomial labels on those faces differ since in particular the dimensions of the faces can be read off by the monomial label.  Once again, from Proposition \ref{prop:cellres} we conclude that the resolution is minimal.
\end{proof}

\begin{example}
In Figure~\ref{fig:ex1} we see a cointerval 2-graph along with its interval representation.  The minimal cellular resolution $X_H$ is depicted in Figure \ref{fig:complex}.  This graph is also the complement of a threshold graph and also appears in \cite{CN1}.
\end{example}

In independent work, Nagel and Reiner \cite{NR} construct cellular resolutions of the edge ideals of strongly stable hypergraphs (among other non-square free classes).  As we mentioned above, the cointerval $d$-graphs form a strictly larger class than strongly stable graphs, and our construction of $X_H$ specializes to the `complex of boxes' developed in \cite{NR}.  For the case $d=2$ it is known that strongly stable 2-graphs correspond to \emph{threshold} graphs.  The complement of a threshold graph is threshold, and threshold graphs are interval graphs, and hence our results are more general already in the case $d = 2$.  In particular, there exist interval graphs which are not threshold, as the next example illustrates.  For further examples of 3-graphs which are cointerval but not strongly stable we refer the reader to the Section \ref{sec:casestudy}.

\begin{figure}[ht]
\begin{center}
 \includegraphics[scale=0.5]{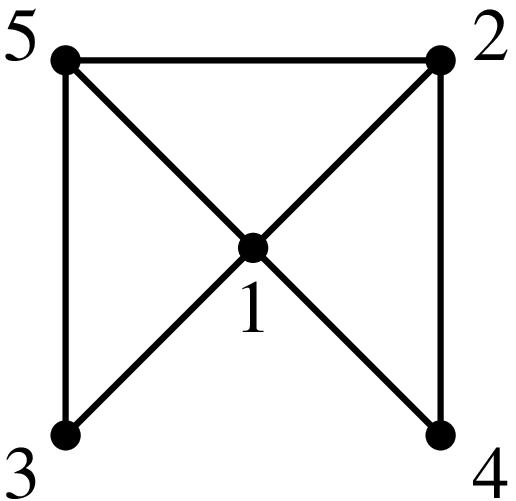} \quad\quad
 \includegraphics[scale=0.5]{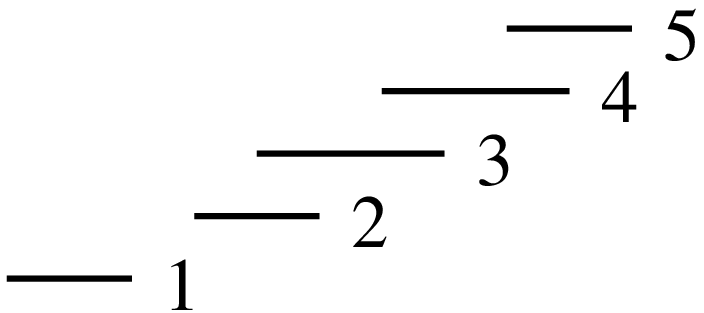}
\caption{A cointerval graph which is not threshold, with its interval representation.}\label{fig:ex2}
\end{center}
\end{figure}

\begin{figure}[ht]
\begin{center}
 \includegraphics[scale=0.5]{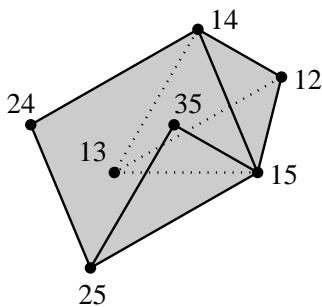}
\caption{The cellular resolution of the graph in Figure~\ref{fig:ex2}.}\label{fig:ex2Res}
\end{center}
\end{figure}

\begin{example}
An example of a graph which is cointerval but not threshold is depicted in
Figure~\ref{fig:ex2}, along with its interval representation (threshold graphs have the property that all induced subgraphs have either a dominating or isolated vertex; here the subgraph induced on $\{2,3,4,5\}$ does not have that property).

Its cellular resolution is depicted in Figure~\ref{fig:ex2Res}, with seven 0-cells, eleven 1-cells, six 2-cells, and a single 3-cell. There are perhaps better ways to illustrate the complex, but we want to emphasize that it is a subcomplex of the subdivision depicted in Figure~\ref{fig:x25} (we will see this in the next section).

For concreteness, we explicitly write down the resolution:

\[ \begin{array}{ccccccccccccc}

0 &  \rightarrow & S &  \rightarrow & S^6 &  \rightarrow & S^{11} &
\rightarrow  & S^7 &  \rightarrow &  I  &  \rightarrow  & 0 \\
&& 11111 && 11110 && 11100 && 11000 \\
&&             &&  11101 && 11010 && 10100 \\
&&             &&  11101 && 11010 && 10010 \\
&&             &&  11011 && 11001 && 10001 \\
&&             &&  11011 && 11001 && 01010 \\
&&             &&  10111 && 10110 && 01001 \\
&&             &&               && 10101 && 00101 \\
&&             &&               && 10101 \\
&&             &&               && 10011 \\
&&             &&               && 01101 \\
&&             &&               && 01011 \\
\end{array} \]

\end{example}

\begin{cor}
If $H$ is a cointerval $d$--graph, then the edge ideal $I_H$ has a $d$--linear resolution.
\end{cor}

\begin{cor}
Let $H$ be a cointerval $d$--graph, and let $I^{\ast}_H$ denote the Alexander dual of $I_H$.  Then the ring $S/I_H^{\ast}$ is Cohen-Macaulay.
\end{cor}

\begin{proof}
This follows from a result of Eagon and Reiner from \cite{ER}.
\end{proof}

The Alexander dual of an edge ideal of a graph $H$ is often called the \emph{vertex cover algebra} of $H$.  The fact that the Alexander duals of cointerval hypergraphs are Cohen-Macaulay has potential applications to face counting of simplicial complexes in the context of algebraic shifting.  Algebraic shifting is a process which associates to an arbitrary simplicial complex $\Delta$ a \emph{shifted} (strongly stable) complex $\Delta^\prime$, preserving much of the combinatorial data of $\Delta$ (see \cite{K}).  Shifted complexes are known to be vertex-decomposable (see \cite{BjKa}) and hence Cohen-Macaulay, and via Stanley-Reisner theory one can conclude certain things about its f-vector (e.g. non-negativity of the h-vector).  The Alexander dual of the independence complex of a hypergraph is a complex whose facets are given by the complements of edges, and we have seen that if $H$ is a cointerval graph this complex is already Cohen-Macaulay.  Hence in general one will not need to `shift as far' to obtain such a complex.

Furthermore, one can ask the question: When are the Alexander duals of cointerval hypergraphs shellable or vertex decomposable?

As was pointed out in \cite{NR}, the cellular complex $X_H$ leads to an easy combinatorial interpretation of the Betti numbers defined in Theorem~\ref{thm:resCoInt}.

\begin{cor}
Let $H$ be a cointerval $d$--graph with vertex set $\{1,2,\ldots, n\}$. The \emph{hyper Ferrers diagram} $F_H$
of $H$ is defined to be
\[ \{ (v_1,v_2, \ldots, v_d) \in \mathbb{N}^d \mid v_1<v_2<\ldots <v_d \textrm{ and }v_1v_2\cdots v_d \in E(H) \}.\]
A \emph{cube} in $\mathbb{N}^d$ is a subset of the type $C=S_1 \times S_2 \times \cdots \times S_d$ where all $S_i \subset \mathbb{N}$.
The \emph{coordinates} of $C$ is the subset of ${\mathbb Z}$ given by $S_1 \cup S_2 \cup \cdots \cup S_d$.

If $V$ is a subset of $\{1,2,\ldots, n\}$ then $\beta_{|V|-d-1,V}(I_H)$ equals the number of cubes in $F_H$ with $V$ as coordinates.
\end{cor}

\begin{proof}
Use the definition of $X_H$ and Theorem~\ref{thm:resCoInt}.
\end{proof}

\begin{rem}\label{rem:otherideals}
In \cite{NR} Nagel and Reiner consider relabelings of their `complex of boxes' to recover minimal cellular resolutions of other classes of monomial ideals.  For them, the class of strongly stable hypergraph edge ideals is denoted $I(K)$, and they consider what they call `depolarizations' to obtain resolutions of subideals of the power of the maximum ideal $\langle x_1, \dots, x_{n-d+1} \rangle ^d$, a class they call $I(M)$.

Furthermore, to an ideal in either of these classes they associate an edge ideal of a $d$-partite graphs, obtaining the classes $I(F(K))$ and $I(F(M))$.  In \cite{NR} it is shown that the same polyhedral complex (with appropriate relabelings) supports minimal resolutions of each of these classes of monomial ideals.  We point that the same constructions can be utilized in our case, with the more general class of cointerval $d$-graphs serving as the `base case'.  We do no work out the details here, although we do say something about the analogue of $I(M)$ in Section \ref{subsec:mixed}.
\end{rem}

\section{Mixed subdivisions and a nice embedding} \label{sec:mixed}

One particularly nice feature of the complexes $X_H$ is that we can give explicit geometric embeddings, without resorting to the high-dimensional ambient space involved in Definition \ref{def:Xsupp}.  It turns out that for the case of complete graphs $K_n^d$ the complex $X_{K_n^d}$ can be realized as a particular \emph{mixed subdivision} of a dilated simplex (definitions below).  As any graph is a subgraph of some complete graph, the general complexes $X_H$ are then subcomplexes of these subdivisions.  This leads to useful geometric representations of our resolutions, and in fact it was these embeddings that led us to the construction of $X_H$ described in the previous section.

\subsection{Mixed subdivisions and the staircase triangulation}

We begin with a brief review of some basic notions of polyhedral geometry (see for example \cite{Z}).  In this section we let $e_1, \dots, e_{k+1}$ denote the standard basis vectors in ${\mathbb R}^{k+1}$ and let $\Delta_k = {\textrm{conv}} \{e_1,\dots,e_{k+1}\}$ denote the standard $k$--simplex.  We fix $d \leq n$ and let $m = n - d$.  We wish to realize ${X}_{K_n^d}$ as a certain mixed subdivision of $d \Delta_m$, the $d$--fold Minkowski sum of an $m$-simplex.

Recall that if $P_1, \dots, P_j$ are polytopes in ${\mathbb R}^{m+1}$, then the \emph{Minkowski sum} is defined to be the polytope
\[P_1 + \cdots +P_j := \big\{x_1 + \cdots + x_j: x_i \in P_i \big\} \subseteq {\mathbb R}^{m+1}.\]
\noindent
Here we will restrict ourselves to the case of $d \Delta_{m}$, the $d$--fold Minkowski sum of $m$-simplices.

To describe our desired subdivisions, we follow \cite{AB} for some definitions and notation.  We define a \emph{fine mixed cell} $X \subseteq d \Delta_{m}$ to be a Minkowski sum $B_1 + \cdots + B_d$, where the $B_i$ are faces of $\Delta_{m}$ which lie in independent affine subspaces, and whose dimensions add up to $m$.  A \emph{fine mixed subdivision} of $d \Delta_{m}$ is then a subdivision of $d \Delta_{m}$ consisting of fine mixed cells.

Now we fix integers $d$ and $n$, let $m= n - d$, and as above let $K^d_n$ denote the complete $d$--graph on $n$ vertices.  We will construct a mixed subdivision $X_{d,n}$ of $d \Delta_{m}$ whose 0-dimensional cells naturally correspond to the vertices of our original complex $X_{K^d_n}$.  For this, it will be convenient to use the following auxiliary construction.  As above we use $\{e_1, \dots, e_{m+1}\}$ to denote the vertices of the simplex $\Delta_{m}$ and consider fine mixed cells of the following kind.  Given a sequence  $(b_1, b_2, \dots, b_{d+1})$ satisfying $1 = b_1 \leq b_2 < \cdots < b_{d} \leq b_{d+1}=m+1$ we let $B_i := \{e_{b_i},e_{b_i+1},\dots,e_{b_{i+1}}\}$ for $1 \leq i \leq d$, and use $(b_1,b_2,\dots,b_{d+1})$ to denote the corresponding (fine) mixed cell $B_1 + B_2 + \cdots + B_d$ of $d \Delta_{m}$.

\begin{example}\label{ex:52}
For $n = 5$, $d=2$, we have $m=3$ so that our complex will be a certain mixed subdivision of $\Delta_3 + \Delta_3$.  The maximal cells of the subdivision are encoded by the sequences $(1,1,4)$, $(1,2,4)$, $(1,3,4)$, and $(1,4,4)$, and each of these correspond to a fine mixed cell depicted in Figure~\ref{fig:k52s}.
\end{example}

\begin{figure}[ht]
\begin{center}
\includegraphics[scale=0.5]{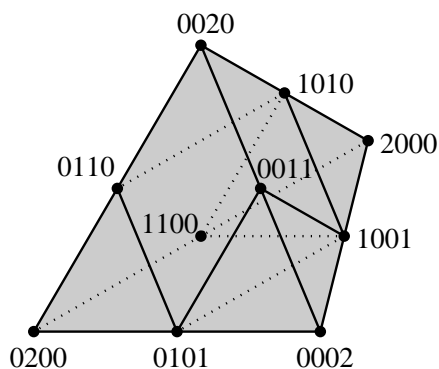}
\caption{Subdivision of $2 \Delta_3$, for $n=5$ and $d=2$.  The vertices are labeled by their realizations as Minkowski sums, so that for instance $0110 = e_2+e_3$.}\label{fig:k52s}
\end{center}
\end{figure}

\begin{example}
For $n = 5$, $d = 3$, we have $m = 2$ so that our complex will be a mixed subdivision of $\Delta_2 + \Delta_2 + \Delta_2$.  The maximal cells of the subdivision are encoded by the sequences $(1,1,1,3)$, $(1,1,2,3)$, $(1,1,3,3)$, $(1,2,2,3)$, $(1,2,3,3)$, and $(1,3,3,3)$, which correspond to the six cells in the subdivision pictured in Figure~\ref{fig:k53s}.
\end{example}

\begin{figure}[ht]
\begin{center}
\includegraphics[scale=0.5]{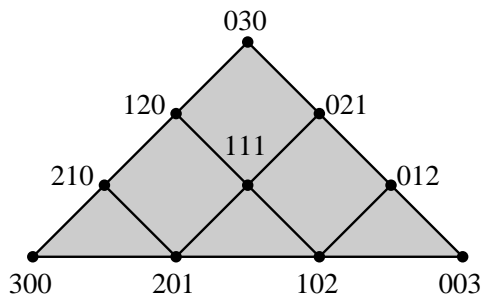}
\caption{Subdivision of $3 \Delta_2$, for $n=5$ and $d=3$.}\label{fig:k53s}
\end{center}
\end{figure}

We claim that the collection of fine mixed cells $\{(b_1,b_2,\dots,b_{d+1}): 1 = b_1 \leq b_2 < \cdots < b_{d} \leq b_{d+1} = m+1 \}$ forms a fine mixed subdivision of the complex $d \Delta_{m}$.  One way to see this is to employ the \emph{Cayley trick}, which (in this special case) gives a bijection between the set of mixed subdivisions of $d \Delta_{m}$ and the \emph{triangulations} of the product of simplices $\Delta_{d-1} \times \Delta_{m}$.  Under this bijection the mixed subdivision that we are describing here can be seen to correspond to the `staircase' triangulation of $\Delta_{d-1} \times \Delta_m$.  We omit the details here, and refer to \cite{AB} and \cite{DJS} for further discussion regarding the Cayley trick and the staircase triangulation.  We record this observation here.

\begin{lemma}
The collection of fine mixed cells $\{(b_1,b_2,\dots,b_m): 1 = b_1 \leq b_2 < \cdots < b_{m-1} \leq b_m = m+1 \}$
forms a fine mixed subdivision of the complex $d \Delta_{m}$.
\end{lemma}

For $m = n - d$, we use $X_{d,n}$ to denote this mixed subdivision.

\subsection{The complexes $X_H$ as mixed subdivisions}\label{subsec:mixed}

As Minkowski sums of the underlying simplex, the vertices of the mixed subdivision $X_{d,n}$ described above are labeled by all monomials of degree $d$ among the variables $\{x_1, \dots, x_{m+1}\}$, where for instance the vertex $e_1 + e_3 + e_3$ is labeled $x_1x_3^2$ (see Figure \ref{fig:k52s}).  In fact these complexes support cellular resolutions of the $d$--th power of the maximal ideal $(x_1, \dots, x_{m+1})^d$ (see \cite{DJS} for a proof of this as well as further discussion).  Here we are interested in the associated squarefree ideal and for this we relabel our complex according to the following well-known bijection between $d$-multisubsets of $[m+1]$ with $d$-subsets of $[m+d]$.

Each monomial of degree $d$ on the vertices $\{x_1, \dots, x_{m+1}\}$ can be thought of as a vector $\alpha \in {\mathbb N}^{m+1}$ with nonnegative coordinates $\alpha_i$ such $\sum \alpha_i = d$.  The nonzero entries of this vector determine a multiset $\{i_1, i_2, \dots, i_d\}$ with $i_1 \leq i_2 \leq \cdots \leq i_d$, where the exponent of $x_j$ gives the number of occurrences of $j$.  To each multiset of this kind we associate a set according to
\[\{i_1, \dots, i_d\} \mapsto \{i_1, i_2+1, \dots, i_d + d - 1 \}.\]
The resulting set has $d$ (distinct) nonzero elements of maximum size $m + 1 + d - 1 = n$, and hence this assignment labels the vertices of $X_{d,m}$ with squarefree monomials corresponding to the edges of the complete hypergraph $K_n^d$.  We note that this relabeling is equivalent to the `polarizations' described by Nagel and Reiner as discussed in Remark \ref{rem:otherideals}.

\begin{example}
In Figure~\ref{fig:x25} we see the complex $X_{2,5}$ from Example~\ref{ex:52}, with vertex labels given by the edges of the graph $K^2_5$.
\end{example}

\begin{figure}[ht]
\begin{center}
\includegraphics[scale=0.5]{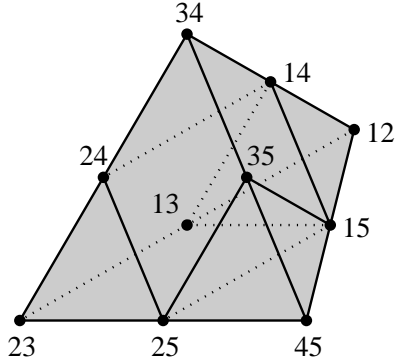}
\caption{The complex $X_{2,5}$, with vertex labels given by the edges of $K^2_5$.}\label{fig:x25}
\end{center}
\end{figure}

Now, if $H$ is a $d$-graph with vertex set $[n]$, we obtain a subcomplex $X_{d,n}[H]$ of the mixed subdivision $X_{d,n}$ by considering the subcomplex induced by those vertices corresponding to the edges of $H$.  We then obtain the following observation.

\begin{prop}\label{prop:embedding}
For any $d$-graph $H$ with vertex set $V(H) = [n]$, the complex $X_{d,n}[H]$ described above is isomorphic as a cell complex to $X_H$.  In particular the complex $X_H$ can be realized as a subcomplex of a mixed subdivision of the dilated simplex $d \Delta_{n-d}$.
\end{prop}

\begin{proof}
We have seen that the vertices of both complexes can be identified.  One checks that this induces a polyhedral isomorphism $X_{d,n}[H] \rightarrow X_H$ which maps a mixed cell $(b_1, b_2, \dots, b_{d+1})$ to $\sigma_1 \times \sigma_2 \times \cdots \times \sigma_d$, where $\sigma_i = \{b_i+i-1,b_i +i, \dots, b_{i+1}+i-1\}$, for $1 \leq i \leq d$.
\end{proof}

\begin{cor}\label{cor:mixedsub}
If $H$ is a cointerval $d$-graph, then the edge ideal $I_H$ has a minimal cellular resolution supported on a subcomplex of a mixed subdivision of a dilated simplex.
\end{cor}

\begin{rem}
In \cite{DJS} it is shown that \emph{any} regular fine mixed subdivision of the dilated simplex $d \Delta_m$ supports a minimal cellular resolution of the ideal $\langle x_1, \dots, x_{m+1} \rangle^d$.  Hence it is a natural question to ask whether any fine mixed subdivision can be used in the construction of resolutions of edge ideals of hypergraphs.  In fact this is not the case, as the following example illustrates.  The particularly well-behaved properties of the staircase subdivision are really necessary here.
\end{rem}

\begin{figure}[ht]
\begin{center}
  \includegraphics[scale=0.5]{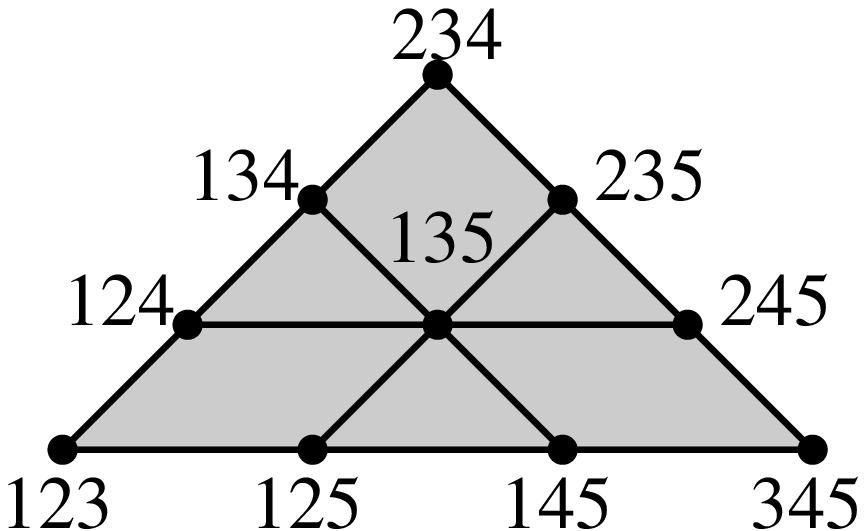} \quad \quad \quad
  \includegraphics[scale=0.5]{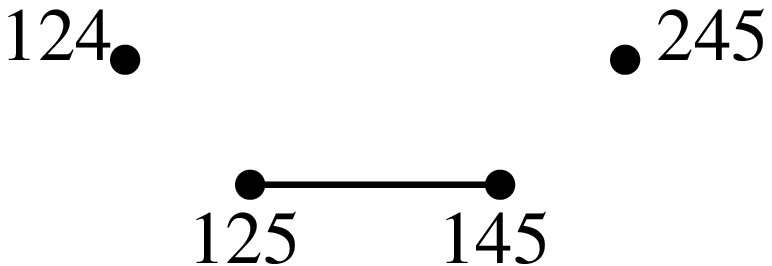}

  \caption{A fine mixed subdivision $Y$, and a disconnected downset $Y_{\leq 1245}$}
  \end{center}
\end{figure}

\section{Constructing resolutions of more general graphs} \label{sec:gluing}

Not all hypergraphs are cointerval (some examples are below) and in this section we discuss methods for building cellular resolutions for more general $d$-graphs.  The basic idea will be to decompose an arbitrary $d$-graph $H$ as a union of cointerval $d$-graphs, and to glue together the associated complexes considered above.

\begin{thm}\label{thm:decompose}
Let $H_1, H_2, \ldots, H_n$ be $d$-graphs on the same vertex set $W$. For all $i=1,2,\ldots, n$ assume that there is a cellular
resolution of $I_{H_1}$ supported by the polyhedral complex $X_i$ with vertices labeled by square-free $\ell_i$.  Assume that the higher dimensional cells are labeled by the least common multiple of their vertices.

If $H=H_1 \cup H_2 \cup \cdots \cup H_n$, $X=X_1 \ast X_2 \ast \cdots \ast X_n$, and $\ell( \sigma )=\textrm{lcm}( \ell_1( \sigma ), \ell_2( \sigma ), \ldots, \ell_n(\sigma) )$, then the complex $X$ with labels $\ell$ supports a cellular resolution of $I_H$.
\end{thm}

\begin{proof}
Let the square-free monomial $\alpha$ support an edge of $H$. We will prove that $X_{\leq \alpha}$ is acyclic.  For this consider
\[  \begin{array}{rcl}
X_{\leq \alpha} & = & \{ \sigma \in X \mid \ell(\sigma) \leq \alpha \} \\
& = & \{ \sigma_1 \ast \cdots \ast \sigma_n \in X \mid   \sigma_1 \in X_1, \ldots,  \sigma_n \in X_n,   \ell(\sigma_1 \ast \cdots \ast \sigma_n) \leq \alpha \} \\
& = & \{ \sigma_1 \ast \cdots \ast \sigma_n \in X \mid   \sigma_1 \in X_1, \ell_1(\sigma_1)\leq \alpha, \ldots,  \sigma_n \in X_n,   \ell_n(\sigma_n) \leq \alpha \} \\
& = & \{ \sigma_1 \ast \cdots \ast \sigma_n \in X \mid   \sigma_1 \in (X_1)_{\leq \alpha}, \ldots, \sigma_2 \in (X_2)_{ \leq \alpha} \} \\
& = & (X_1)_{\leq \alpha} \ast \cdots \ast (X_n)_{\leq \alpha}
\end{array} \]
At least one of the $(X_i)_{\leq \alpha}$ is non-empty, and thus acyclic.  Hence we conclude that $X_{\leq \alpha}$ is also acyclic.
\end{proof}

The most basic example of Theorem~\ref{thm:decompose} recovers what is known as the `Taylor resolution' of $I_H$.  For this note that a hypergraph $H_i$ consisting of a single edge has a cellular resolution supported by a point. The join of $|E(H)|$ points is a $(|E(H)|-1)$--dimensional simplex supporting the resolution of $I_H$.

\begin{defn}
The \emph{linear width} of a $d$--graph $H$, denoted $\omega_{\tt lin}(H)$, is the smallest number $k$ such that $H=H_1 \cup H_2 \cup \cdots \cup H_k$, with each $H_i$ a cointerval $d$--graph.
\end{defn}

The linear width is well-defined and $\omega_{\tt lin}(H) \leq |E(H)|$ since any hypergraph with one edge is a cointerval hypergraph.
We have chosen the name \emph{linear width} since for 2-graphs it is closely related to the path-width \cite{RS} and band-width \cite{Bou,Feige},
and if the linear width of $H$ is one, the ideal $I_H$ has a linear resolution.

Bourgain \cite{Bou} and Feige \cite{Feige} have developed rather general theories regarding modifying combinatorial objects to obtain `perfect elimination orders'.  If these ideas apply to decomposing hypergraphs into cointerval hypergraphs, then we expect the linear width to grow rather slowly.  In fact we conjecture that for any a fixed $d$ there is a constant $C_d$ such that $\omega_{\tt lin}(H) < C_d n$ for any $d$-graph on $n$ vertices.  A solution to this conjecture would give new general bounds on graded Betti numbers of hypergraph edge ideals.

In this paper all monomial ideals are generated in a fixed degree $d$, but there is a generalization of the previous theorem to the corresponding non-uniform hypergraph case, since we never used that the edges are of the same order in the proof.

\subsection{A case study: 3-graphs on at most 5 vertices}\label{sec:casestudy}

In this section we study (unlabeled) 3-graphs on at most 5 vertices. There are 34 of them.  With an exhaustive computer search we find that 26 of these are cointerval under suitable labelings (the first 26 in the list below), 10 of which are \emph{not} strongly stable (graphs 7,10,11,17,19,21,22,23,25,26).  The number of strongly stable graphs (16) is verified by the enumerative results presented in Theorem 3 of \cite{Kli}.

\subsubsection{The cointerval 3-graphs on 5 vertices}

\begin{figure}
\begin{center}
\makebox[12cm][c]{
\begin{tabular}{ccccc}
\includegraphics[scale=0.75]{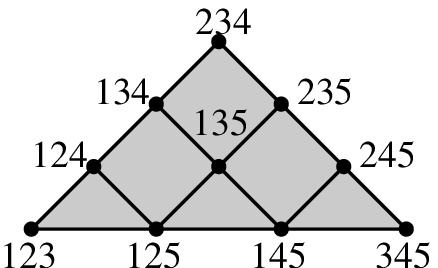} &
\includegraphics[scale=0.75]{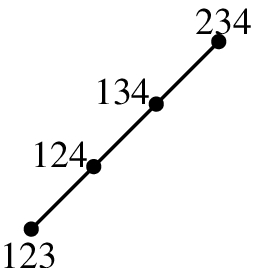} &
\includegraphics[scale=0.75]{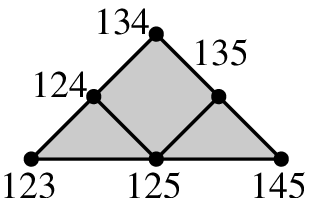} &
\includegraphics[scale=0.75]{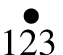} &
\includegraphics[scale=0.75]{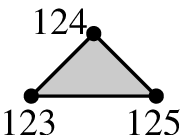} \\
Graph 2 & Graph 3 & Graph 4 & Graph 5 & Graph 6 \\
\\
\includegraphics[scale=0.75]{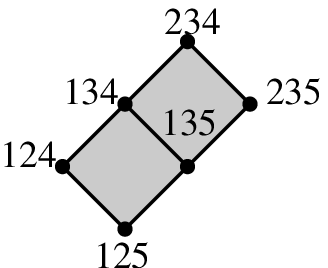} &
\includegraphics[scale=0.75]{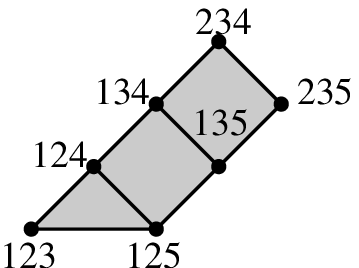} &
\includegraphics[scale=0.75]{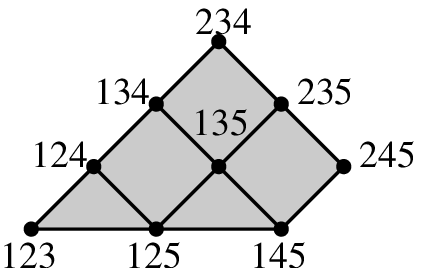} &
\includegraphics[scale=0.75]{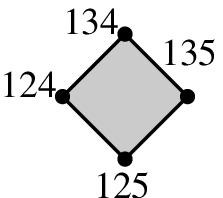} &
\includegraphics[scale=0.75]{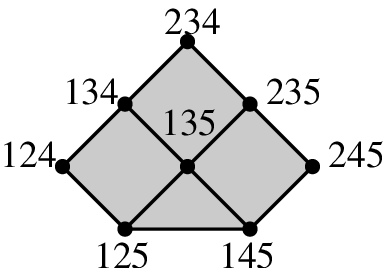} \\
Graph 7 & Graph 8 & Graph 9 & Graph 10 & Graph 11 \\
\\
\includegraphics[scale=0.75]{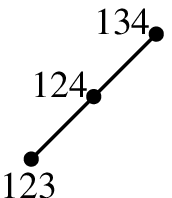} &
\includegraphics[scale=0.75]{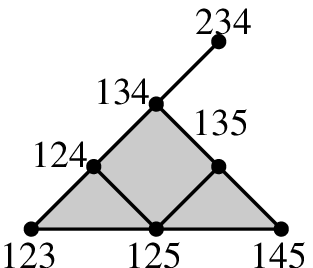} &
\includegraphics[scale=0.75]{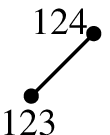} &
\includegraphics[scale=0.75]{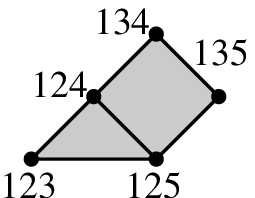} &
\includegraphics[scale=0.75]{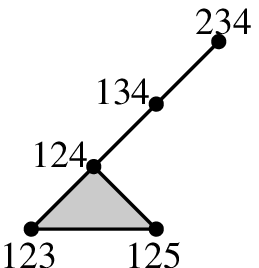} \\
Graph 12 & Graph 13 & Graph 14 & Graph 15 & Graph 16 \\
\\
\includegraphics[scale=0.75]{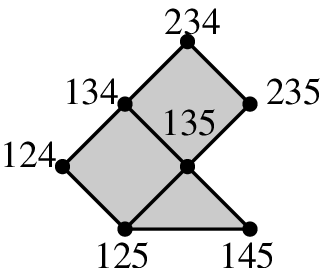} &
\includegraphics[scale=0.75]{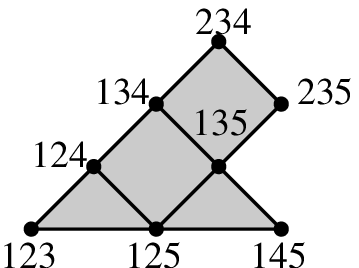} &
\includegraphics[scale=0.75]{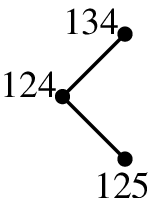} &
\includegraphics[scale=0.75]{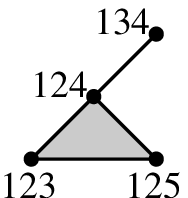} &
\includegraphics[scale=0.75]{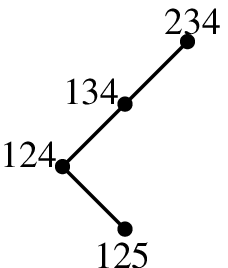} \\
Graph 17 & Graph 18 & Graph 19 & Graph 20 & Graph 21 \\
\\
\includegraphics[scale=0.75]{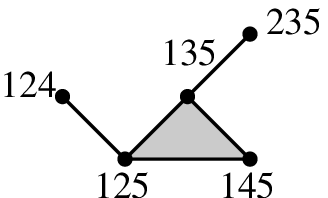} &
\includegraphics[scale=0.75]{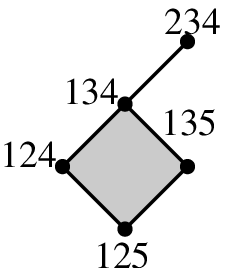} &
\includegraphics[scale=0.75]{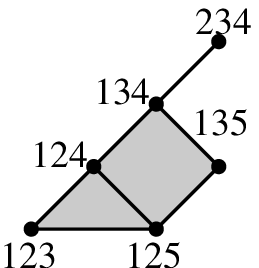} &
\includegraphics[scale=0.75]{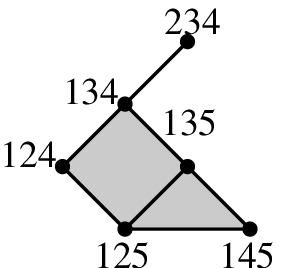} &
\includegraphics[scale=0.75]{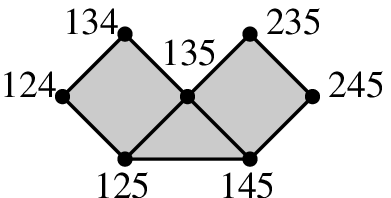} \\
Graph 22 &Graph  23 & Graph 24 & Graph 25 &Graph  26 \\
\end{tabular}}
\caption{Cellular minimal resolutions for cointerval graphs 2-26.}\label{fig:min35}
\end{center}
\end{figure}

In Figure~\ref{fig:min35} we see minimal cellular resolutions of the cointerval 3-graphs on at most 5 vertices (Graph 1 is the empty graph).  The graphs themselves can of course be recovered by recording the labels on the 0-cells.

\begin{figure}[ht]
\begin{center}
\makebox[0cm][c]{
\begin{tabular}{cc}
\includegraphics[scale=0.75]{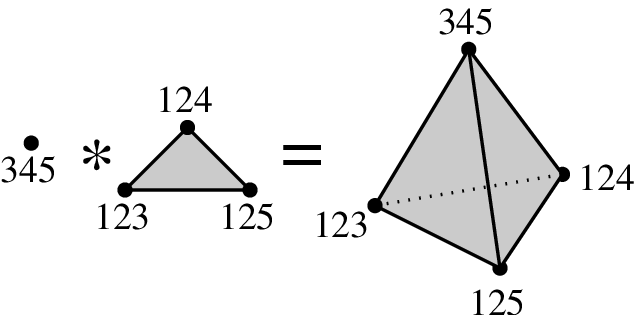} &
\includegraphics[scale=0.75]{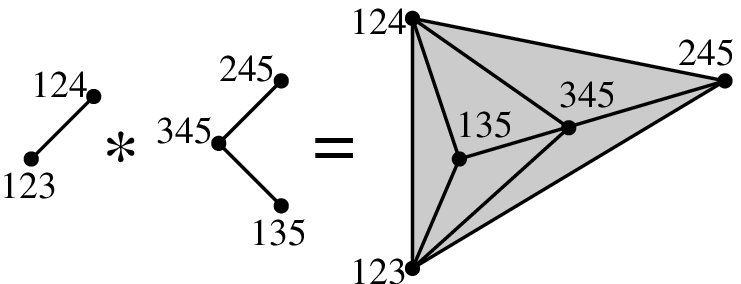} \\
Graph 27 & Graph 28\\
\\
\includegraphics[scale=0.75]{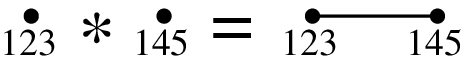} &
\includegraphics[scale=0.75]{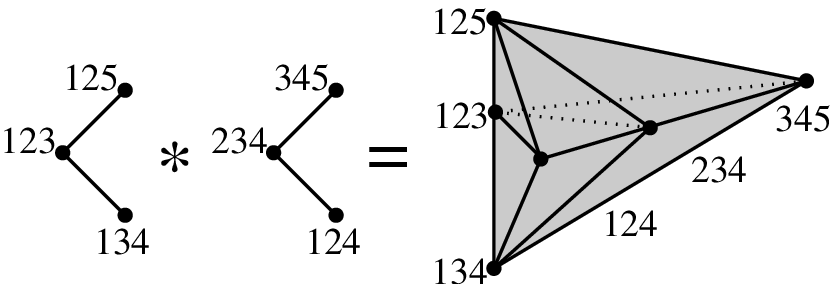} \\
Graph 29 & Graph 30\\
\\
\includegraphics[scale=0.75]{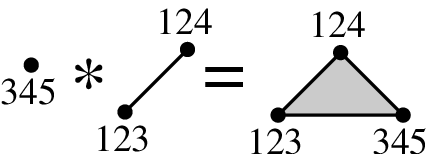} &
\includegraphics[scale=0.75]{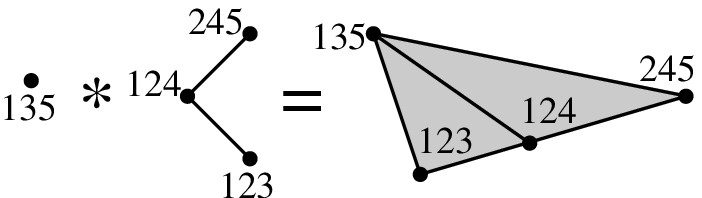} \\
Graph 31 & Graph 32\\
\\
\includegraphics[scale=0.75]{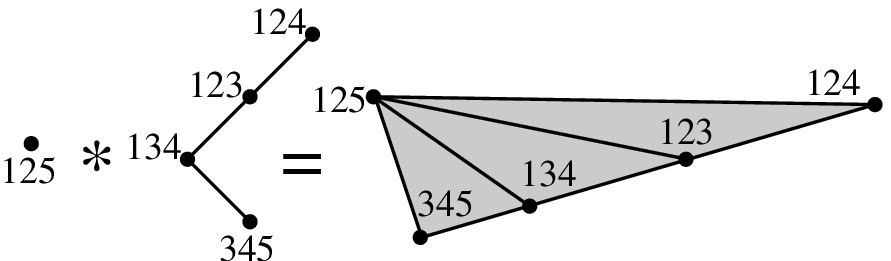} &
\includegraphics[scale=0.75]{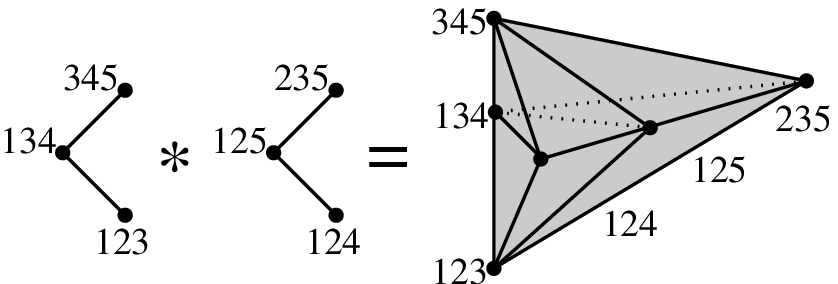} \\
Graph 33 & Graph 34\\
\end{tabular}}
\caption{Cellular resolutions for graphs 27-34 using Theorem~\ref{thm:decompose}}\label{fig:joins35}
\end{center}
\end{figure}

\subsubsection{The non-cointerval 3-graphs on 5 vertices.}

Next we turn to decompositions of non-cointerval graphs. Each of the graphs 27-34 are presented in Figure~\ref{fig:joins35} with a cellular resolution constructed as a join of minimal resolutions.  We point out that using only strongly stable subgraphs in a decomposition, graph 28 needs to be decomposed into three graphs.

\section{Further questions}\label{sec:further}

\subsection{A larger class of graphs}
As we have seen, for any $d$-graph $H$ the complex $X_H$ that we construct has the property that the dimension of a face $F$ is given by $i - d$, where $i$ is the total degree of the monomial label on $F$.  Hence whenever $X_H$ supports a resolution of $I_H$, it is $d$-linear.  It is a well known result of Fr\"{o}berg (see \cite{F}) that a 2-graph $H$ has a 2-linear resolution if and only $H$ is the complement of a \emph{chordal} graph.  Interval graphs (which correspond to cointerval 2-graphs) are a proper subset of chordal graphs, and in particular there exist graphs which are chordal but not interval. These include the graphs in Figure~\ref{fig:cNotI}.

\begin{figure}[ht]
\begin{center}
\includegraphics[width=80mm]{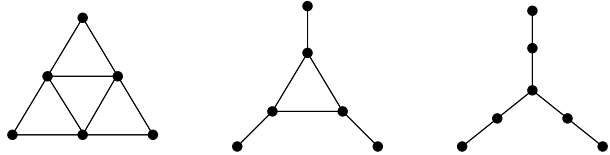}
\caption{Three graphs which are chordal but not interval.}\label{fig:cNotI}
\end{center}
\end{figure}

A natural question to ask is whether our complexes $X_H$ can be used to obtain resolutions of a more general class of graphs.  It turns out that our construction will \emph{not} work for the class of complements of chordal graphs.  In fact, if $G$ is taken to be the complement of the first graph in Figure~\ref{fig:cNotI} (which happens to be isomorphic to the second graph in that list), one can check that no labeling of the vertices with $\{1, \dots, 6\}$ induces a complex $X_G$ which supports a resolution.  However, it is still an open question to determine the largest class of graphs for which our construction do apply.  We note that the classes recently defined by Emtander \cite{Em} and Woodroofe \cite{W} could be good candidates.

\subsection{Functoriality and more general complexes}
Suppose that $G$ and $H$ are graphs on vertex sets $[m]$ and $[n]$, respectively.  One can check that if $f:G \rightarrow H$ is a directed graph homomorphism then there is an induced polyhedral map $f_*: X_G \rightarrow X_H$.  Furthermore, the map $f$ gives rise to a map $f:S_m \rightarrow S_n$, where $S_j := k[x_1,\dots, x_j]$, and hence gives $S_n$ (and in turn $I_H$) the structure of an $S_m$-module.  The polyhedral map $f_*$ then gives rise to a map of chain complexes of $S_m$-chain complexes.  This functoriality then gives rise to the possibility of applications, where for instance algebraic invariants such as Betti numbers can used to produce obstructions to the existence of graph homomorphisms, in the spirit of equivariant obstructions in the context of $\textrm{Hom}$ complexes.

As we discussed in Section \ref{sec:complex} the complexes $X_H$ can be viewed as special case of a more general complex of homomorphisms between directed graphs.  For this, suppose $T$ and $H$ are graphs with vertex sets $[m]$ and $[n]$, respectively.  The complex $X_{T,H} = \textrm{Hom}(T,H)$  parameterizes directed homomorphisms and, as above, give rise to a chain complex.  In this general case, the entries of the complex should no longer be considered as modules over the polynomial ring, but instead as modules over the DG-algebra $\textrm{Hom}(E,T)$ (where, as above, $E$ is the directed $m$-edge).  We see further development in this area as a subject for future work.

\end{document}